\documentclass[11pt,a4paper]{article}
\usepackage[T1]{fontenc}
\usepackage[utf8]{inputenc}
\usepackage{lmodern}
\usepackage{amsmath,amssymb,amsthm,mathtools,mathrsfs}
\usepackage[margin=28mm,headheight=14pt]{geometry}
\usepackage{microtype}
\usepackage{enumitem}
\usepackage{needspace}
\usepackage{xcolor}
\usepackage{fancyhdr}
\usepackage{hyperref}
\hypersetup{colorlinks=true,linkcolor=black,citecolor=black,urlcolor=blue!45!black,
 pdftitle={A short proof of the Bernoulli decomposition theorem},
 pdfauthor={Witold Bednorz, Rafal Martynek, Rafal Meller},
 pdfsubject={Revised manuscript: truncations, partitions, and Bernoulli processes}}

\allowdisplaybreaks[1]
\numberwithin{equation}{section}
\newtheorem{theorem}{Theorem}[section]
\newtheorem{proposition}[theorem]{Proposition}
\newtheorem{lemma}[theorem]{Lemma}

\theoremstyle{remark}
\newtheorem{remark}[theorem]{Remark}
\newcommand{\E}{\mathbb E}
\newcommand{\Prob}{\mathbb P}
\newcommand{\N}{\mathbb N}
\newcommand{\Z}{\mathbb Z}
\newcommand{\R}{\mathbb R}
\newcommand{\A}{\mathcal A}
\newcommand{\Pset}{\mathcal P}
\newcommand{\diam}{\operatorname{diam}}
\newcommand{\dist}{\operatorname{dist}}

\newcommand{\norm}[1]{\left\lVert #1\right\rVert}

\newcommand{\clip}[2]{\left[#1\right]_{#2}}
\newcommand{\F}{\mathcal F}
\newcommand{\eps}{\varepsilon}

\title{\vspace{-1em}A short proof of the Bernoulli\texorpdfstring{\\}{ }decomposition theorem}
\author{{Witold Bednorz, Rafal Martynek and Rafal Meller}
\footnote{Research partially supported by  Grant UMO-2022/47/B/ST1/02114}
\footnote{Institute of Mathematics, University of Warsaw, Banacha 2, 02-097 Warszawa, Poland}}
\date{}

\begin{document}
\maketitle
\thispagestyle{plain}
\begin{abstract}
We give a construction of admissible partitions that yields the Bernoulli
decomposition theorem. The construction uses clipped coordinates and a
splitting into two monotone functions. The canonical distance of the split
process is comparable to the distance before splitting. This permits a direct
application of the coordinate-removal principle of Bednorz and Lata{\l}a.
A single selection of scales reduces the chaining sum to decreases of the
associated functionals over intervals of overlap at most two. Combined with
the deterministic partition criterion of Bednorz and Lata{\l}a, this gives
sets $T_1,T_2$ such that $T\subset T_1+T_2$ and
$\sup_{t\in T_1}\norm{t}_1+\gamma_2(T_2)\le C b(T)$.
\end{abstract}

\section{Introduction}

Let $(\eps_i)_{i\ge1}$ be independent random variables taking the values
$-1$ and $1$ with equal probabilities. For $t\in\ell^2$, set
\[
 X_t=\sum_{i\ge1}t_i\eps_i,
 \qquad b(T)=\E\sup_{t\in T}X_t,\qquad T\subset\ell^2.
\]
For an arbitrary index set, the expectation of the supremum is understood as
the supremum of the corresponding expectations over finite subsets. For
countable sets this agrees with the usual expectation. All constructions
below are carried out for countable sets; the passage to arbitrary sets is
given in Section~\ref{sec:mainproof}.

Write $N_n=2^{2^n}$ for $n\ge0$. We use the admissible-net convention
\begin{equation}\label{eq:gamma}
 \gamma_2(S)=\inf_{(S_n)}\sup_{s\in S}
       \sum_{n\ge0}2^{n/2}\dist_2(s,S_n),
 \qquad |S_0|=1,\quad 1\le |S_n|\le N_n\quad(n\ge1),
\end{equation}
where the finite sets $S_n$ may lie in the ambient Hilbert space. This
convention is equivalent, up to universal constants, to the usual definitions
using nets in $S$ or admissible partitions; see~\cite{Talagrand2014}.

The following theorem was proved by Bednorz and Lata{\l}a~\cite{BL2014}.
\begin{theorem}[Bernoulli decomposition]\label{thm:main}
For every nonempty $T\subset\ell^2$ with $b(T)<\infty$, there are
$T_1,T_2\subset\ell^2$ such that
\begin{equation}\label{eq:main}
 T\subset T_1+T_2,\qquad
 \sup_{x\in T_1}\norm{x}_1\le Cb(T),\qquad
 \gamma_2(T_2)\le Cb(T),
\end{equation}
where $C$ is a universal constant.
\end{theorem}

The converse estimate follows from the triangle inequality for suprema,
the bound $b(T_1)\le\sup_{x\in T_1}\norm{x}_1$, and the generic chaining
upper bound $b(T_2)\le C\gamma_2(T_2)$. Thus~\eqref{eq:main} characterizes
boundedness of Bernoulli processes.

Our aim is to simplify the partition construction in the proof of
Theorem~\ref{thm:main}. We retain the deterministic decomposition criterion
and the coordinate-removal principle from~\cite{BL2014}. The latter is
recalled, with a proof, in Lemma~\ref{lem:removal}. Instead of a family of
chopping maps depending on several truncation levels, we use a single clipped
coordinate at each node. A change of center splits this coordinate into two
monotone pieces. Keeping both pieces preserves the old metric within a
factor $\sqrt2$, while keeping only the finer piece decreases the expected
supremum. A second subdivision reconciles these two operations.

The remaining issue is to sum the scales along each branch. An ordering
lemma selects indices at which an unchanged level is followed by an increase.
The next increase occurs within a fixed number of steps. Each selected index
then produces a functional decrease, and the intervals on which these
decreases occur have overlap at most two.

Throughout, constants described as universal are independent of $T$ and of
the partition indices. Constants $C$ without a subscript may change between
occurrences. The parameters $K,r,\kappa$ will be fixed, in this order, in
Section~\ref{sec:parameters}; they will also be universal. All balls are closed.
Empty partition elements are discarded.

\section{A deterministic partition criterion}\label{sec:criterion}

A refining sequence $(\A_n)_{n\ge0}$ of partitions of $T$ is
\emph{admissible} if $\A_0=\{T\}$ and $|\A_n|\le N_n$ for $n\ge1$.
For $t\in T$, let $A_n(t)$ be its member of $\A_n$. Data assigned to a
partition element are also regarded as functions of $t$; for example,
$j_n(t)=j_n(A_n(t))$.

We use the following result in the form of~\cite[Theorem~3.1]{BL2014}.
The change from the partition convention for $\gamma_2$ to~\eqref{eq:gamma}
only changes the universal constant.

\begin{theorem}[Partition criterion]\label{thm:criterion}
Let $r\ge2$, $M>0$, and let $(\A_n)$ be admissible. Assign an integer
$j_n(A)$ and a point $\pi_n(A)\in T$ to every $A\in\A_n$. Suppose
\begin{equation}\label{eq:initialdiam}
 \diam_2(T)^2\le M r^{-2j_0(T)}.
\end{equation}
For $n\ge1$, $A\in\A_n$, and its parent $A'\in\A_{n-1}$, assume that
either
\begin{enumerate}[label=\textup{(\roman*)},leftmargin=2em]
\item $j_n(A)=j_{n-1}(A')$ and $\pi_n(A)=\pi_{n-1}(A')$; or
\item $j_n(A)>j_{n-1}(A')$, $\pi_n(A)\in A'$, and, for every $t\in A$,
\begin{equation}\label{eq:criterionenergy}
 \sum_{i\in I_n(t)}\min\{(t_i-\pi_n(t)_i)^2,r^{-2j_n(t)}\}
 \le M2^n r^{-2j_n(t)},
\end{equation}
where
\begin{equation}\label{eq:tracking}
 I_n(t)=\{i\ge1:|\pi_{k+1}(t)_i-\pi_k(t)_i|
                 \le r^{-j_k(t)}\text{ for }0\le k<n\}.
\end{equation}
\end{enumerate}
Then $T\subset T_1+T_2$ for some $T_1,T_2\subset\ell^2$ satisfying
\begin{align}
 \sup_{x\in T_1}\norm{x}_1
 &\le CM\sup_{t\in T}\sum_{n\ge0}2^n r^{-j_n(t)},\label{eq:criterionl1}\\
 \gamma_2(T_2)
 &\le C\sqrt M\sup_{t\in T}\sum_{n\ge0}2^n r^{-j_n(t)}.\label{eq:criteriongamma}
\end{align}
\end{theorem}

Our construction has $j_{n+1}-j_n\in\{0,1\}$ and satisfies a clipped
quadratic estimate stronger than~\eqref{eq:criterionenergy}, with $M=r^2$.
It remains to arrange
\begin{equation}\label{eq:targetsum}
 \sup_{t\in T}\sum_{n\ge0}2^n r^{-j_n(t)}\le Cb(T).
\end{equation}

\section{Bernoulli estimates}\label{sec:inputs}

For a countable set $B$ and a map $z:B\to\ell^2(\Lambda)$, write
\[
 b_z(D)=\E\sup_{t\in D}\sum_{i\in\Lambda}z_i(t)\eps_i,
 \qquad \varnothing\ne D\subset B.
\]
This functional is nonnegative, is monotone under inclusion, and is unchanged
by adding a fixed vector to $z$. Removing coordinates can only decrease it,
by conditional Jensen's inequality.

We recall the classical Bernoulli comparison, concentration, and Sudakov
estimates in the forms recorded in~\cite[Theorems~2.2, 2.5, 2.7 and
Proposition~2.8]{BL2014}. If $q_{i,l},p_i:\R\to\R$ satisfy
\begin{equation}\label{eq:contractioncondition}
 \sum_l|q_{i,l}(x)-q_{i,l}(y)|\le|p_i(x)-p_i(y)|
 \quad(i\in\Lambda;\ x,y\in\R),
\end{equation}
then, whenever the coefficient vectors are square summable,
\begin{equation}\label{eq:contraction}
 \E\sup_{t\in B}\sum_{i,l}q_{i,l}(t_i)\eps_{i,l}
 \le\E\sup_{t\in B}\sum_i p_i(t_i)\eps_i.
\end{equation}
This is the multiple-coordinate form of the contraction principle
\cite[Corollary~2.4]{BL2014}.

If $S=\sup_{t\in B}(a_t+\sum_i z_i(t)\eps_i)$ is almost surely finite,
where $(a_t)$ is deterministic and $\sup_t\norm{z(t)}_2\le\sigma$, then
\begin{equation}\label{eq:concentration}
 \Prob\{|S-\E S|\ge u\}\le2\exp\left(-\frac{u^2}{C\sigma^2}\right)
 \qquad(u>0).
\end{equation}
When $\sigma=0$, the assertion is understood in its deterministic sense.
In particular, if $Y_t^1,\ldots,Y_t^m$ are independent copies of a Bernoulli
process with $\sup_t\norm{Y_t^1}_{L^2}\le\sigma$, and $H$ is independent
of these copies, conditioning on $H$ and integrating the union bound give
\begin{equation}\label{eq:copies}
 \E\max_{p\le m}\sup_{t\in B}(H_t+Y_t^p)
 \le\E\sup_{t\in B}(H_t+Y_t^1)+C\sigma\sqrt{\log m},\qquad m\ge2.
\end{equation}

There are universal constants $c_{\rm s}>0$ and $C_{\rm s}>0$ with the
following properties. If $z^1,\ldots,z^m$ have pairwise $\ell^2$ distances
at least $a$ and $\norm{z^p-z^q}_\infty\le\beta$ for all $p,q$, then
\begin{equation}\label{eq:sudakov}
 \E\max_{p\le m}\sum_i z_i^p\eps_i
 \ge c_{\rm s}\min\{a\sqrt{\log m},a^2/\beta\}.
\end{equation}
We use the translation-invariant formulation; it follows by subtracting
one of the vectors. If instead $\norm{z^p}_\infty\le\beta$ and
$H_p\subset B_2(z^p,\sigma)$ are nonempty, then
\begin{equation}\label{eq:growth}
 b\left(\bigcup_{p\le m}H_p\right)
 \ge c_{\rm s}\min\{a\sqrt{\log m},a^2/\beta\}
       -C_{\rm s}\sigma\sqrt{\log m}+\min_{p\le m}b(H_p).
\end{equation}
Decreasing $c_{\rm s}$ if necessary makes both statements valid with the
same constant.

\subsection{Removing coordinates}

The next lemma is a reformulation of~\cite[Proposition~2.10]{BL2014}.
We include a proof to specify the centering and the conditional expectations
used in the selection of random points.

\begin{lemma}[Coordinate removal]\label{lem:removal}
Let $B$ be a nonempty countable set and let
$x:B\to\ell^2(\Lambda_1)$ and $y:B\to\ell^2(\Lambda_2)$, where the
coordinate sets are disjoint. Write $z=(x,y)$ and
$d(s,t)=\norm{z(s)-z(t)}_2$. Suppose $b_z(B)<\infty$ and
\[
 \sup_{s,t\in B}\norm{z(s)-z(t)}_\infty\le\beta,
 \qquad \sup_{s,t\in B}\norm{y(s)-y(t)}_2\le\delta.
\]
For $m\ge2$ and $\rho>0$, there are $t^1,\ldots,t^m\in B$ such that,
with $R=B\setminus\bigcup_{p\le m}B_d(t^p,\rho)$, either $R$ is empty or
\begin{equation}\label{eq:removal}
 b_y(R)\le b_z(B)
 -c_*\min\{\rho\sqrt{\log m},\rho^2/\beta\}
 +C_*\delta\sqrt{\log m}.
\end{equation}
Here $c_*,C_*>0$ are universal constants and $\beta>0$ may be any bound
as above.
\end{lemma}

\begin{proof}
If $m$ balls cover $B$, there is nothing to prove. Assume otherwise, so
every remainder appearing below is nonempty. Put
\[
 L_0=\min\{\rho\sqrt{\log m},\rho^2/\beta\},\qquad
 a_* =\inf_{t^1,\ldots,t^m\in B}
 b_y\left(B\setminus\bigcup_{p\le m}B_d(t^p,\rho)\right).
\]
If $\delta\ge\rho/2$, the conclusion follows from
$b_y(R)\le b_y(B)\le b_z(B)$ by choosing $C_*$ sufficiently large.
Hence suppose $\delta<\rho/2$.

Fix $w\in B$ and subtract $z(w)$ from all coefficient vectors. This does
not change any distance or expected supremum and ensures
$\sup_t\norm{y(t)}_2\le\delta$. Define
\[
 H_t=\sum_{i\in\Lambda_1}x_i(t)\eps_i,
 \qquad Y_t^p=\sum_{i\in\Lambda_2}y_i(t)\eps_i^p,
 \qquad 1\le p\le m,
\]
using independent families of signs. By~\eqref{eq:copies},
\begin{equation}\label{eq:removalupper}
 V:=\E\max_{p\le m}\sup_{t\in B}(H_t+Y_t^p)
 \le b_z(B)+C\delta\sqrt{\log m}.
\end{equation}

Let $\mathscr G_p$ be generated by the first $p$ copies of $Y$, and fix
$\xi>0$. Inductively set
\[
 R_p=B\setminus\bigcup_{l<p}B_d(t^l,\rho),\qquad
 S_p=\sup_{t\in R_p}Y_t^p,
\]
and choose a $\mathscr G_p$-measurable $t^p\in R_p$ with
$Y_{t^p}^p\ge S_p-\xi$. A fixed enumeration of $B$ gives such a measurable
choice. The set $R_p$ is $\mathscr G_{p-1}$-measurable and is independent of
the $p$th copy. Therefore
\[
 \mu_p:=\E(S_p\mid\mathscr G_{p-1})=b_y(R_p)\ge a_*.
\]
The last inequality follows by adding further balls to obtain a remainder
of $m$ balls contained in $R_p$. Applying~\eqref{eq:concentration}
conditionally on $\mathscr G_{p-1}$, and then using a union bound, gives
\begin{equation}\label{eq:adaptiveconcentration}
 \E\max_{p\le m}(\mu_p-S_p)_+
 \le C\delta\sqrt{\log m}.
\end{equation}

Conditionally on all copies of $Y$, the points $t^p$ are fixed and the
process $H$ is unchanged. Their full distances exceed $\rho$, while their
$y$-distances are at most $\delta$. Hence their $x$-distances are at least
$\rho/2$. By~\eqref{eq:sudakov},
\[
 \E_H\max_{p\le m}H_{t^p}\ge\frac{c_{\rm s}}4 L_0.
\]
It follows from~\eqref{eq:adaptiveconcentration} that
\[
 V\ge\E\max_{p\le m}(H_{t^p}+Y_{t^p}^p)
 \ge\frac{c_{\rm s}}4L_0+a_*-C\delta\sqrt{\log m}-\xi.
\]
Combine this with~\eqref{eq:removalupper} and let $\xi\downarrow0$.
Finally choose deterministic centers whose remainder has $b_y$-value within
$c_{\rm s}L_0/8$ of $a_*$. This gives~\eqref{eq:removal} with
$c_*=c_{\rm s}/8$ and a suitable $C_*$.
\end{proof}

\section{Truncations and subdivisions}\label{sec:truncation}

For $r\ge2$ and $j\in\Z$, define
\[
 \clip{x}{j}=\max\{-3r^{-j},\min\{x,3r^{-j}\}\}.
\]
For $I\subset\N$, $u\in\ell^2$, and nonempty $D\subset\ell^2$, put
\begin{align}
 \varphi_{I,u,j}(t)&=(\clip{t_i-u_i}{j})_{i\in I},\notag\\
 F_{I,u,j}(D)&=\E\sup_{t\in D}\sum_{i\in I}\clip{t_i-u_i}{j}\eps_i,
 \label{eq:functional}\\
 d_{I,u,j}(s,t)&=\norm{\varphi_{I,u,j}(s)-\varphi_{I,u,j}(t)}_2.
 \label{eq:metric}
\end{align}
These coefficient vectors are square summable. Contraction and translation
invariance give
\begin{equation}\label{eq:boundedfunctional}
 0\le F_{I,u,j}(D)\le b(D).
\end{equation}
The distances in~\eqref{eq:metric} may be pseudometrics, which causes no
difficulty in the constructions below.

\subsection{Splitting a clipped coordinate}

\begin{lemma}\label{lem:split}
Fix $I,u,j$ and $v\in\ell^2$, and set
\[
 J=\{i\in I:|v_i-u_i|\le r^{-j}\}.
\]
For $i\in J$, write
\[
 f_i(s)=\clip{s-u_i}{j},\qquad
 g_i(s)=\clip{s-v_i}{j+1},\qquad \theta_i(s)=f_i(s)-g_i(s).
\]
Both $g_i$ and $\theta_i$ are nondecreasing. If
\begin{equation}\label{eq:splitmap}
 \Psi(t)=\big((f_i(t_i))_{i\in I\setminus J},
                 (\theta_i(t_i))_{i\in J},(g_i(t_i))_{i\in J}\big),
\end{equation}
where $f_i(s)=\clip{s-u_i}{j}$ also for $i\in I\setminus J$, and
$\bar d(s,t)=\norm{\Psi(s)-\Psi(t)}_2$, then
\begin{equation}\label{eq:metriccomparison}
 \tfrac12 d_{I,u,j}(s,t)^2\le\bar d(s,t)^2
                \le d_{I,u,j}(s,t)^2.
\end{equation}
For every nonempty countable $D$,
\begin{equation}\label{eq:splitcomparison}
 F_{J,v,j+1}(D)\le b_\Psi(D)\le F_{I,u,j}(D).
\end{equation}
Moreover,
\begin{equation}\label{eq:energytransfer}
 \norm{\varphi_{J,v,j+1}(t)}_2^2\le d_{I,u,j}(t,v)^2.
\end{equation}
\end{lemma}

\begin{proof}
The interval $[v_i-3r^{-j-1},v_i+3r^{-j-1}]$ is contained in
$[u_i-3r^{-j},u_i+3r^{-j}]$, because $r\ge2$ and
$|v_i-u_i|\le r^{-j}$. Thus $g_i$ is increasing only where $f_i$ has slope
one, and $f_i-g_i$ is nondecreasing.

For $s\ge t$, the two increments $g_i(s)-g_i(t)$ and
$\theta_i(s)-\theta_i(t)$ are nonnegative and sum to $f_i(s)-f_i(t)$.
Hence
\[
 \tfrac12(f_i(s)-f_i(t))^2
 \le (g_i(s)-g_i(t))^2+(\theta_i(s)-\theta_i(t))^2
 \le(f_i(s)-f_i(t))^2.
\]
Summing proves~\eqref{eq:metriccomparison}. The same observation gives
\eqref{eq:contractioncondition}, and therefore the second inequality
in~\eqref{eq:splitcomparison}. The first follows by removing the first two
blocks of $\Psi$. Finally $g_i(v_i)=0$, so~\eqref{eq:energytransfer}
follows from the upper bound in~\eqref{eq:metriccomparison}.
\end{proof}

In particular, the comparison of functionals in
\eqref{eq:splitcomparison} holds on every descendant subset, regardless of
whether the new center belongs to that subset.

\subsection{The first subdivision}

\begin{lemma}[Greedy subdivision]\label{lem:greedy}
There are universal constants $K\ge1$ and $\eta_{\rm G}>0$ with the
following property. Fix $q\ge1$, $I,u,j$, and a nonempty countable set $A$
with $F_{I,u,j}(A)<\infty$. Let $m=N_{q-1}$, or, when $q\ge2$, let
$m=N_{q-2}$. Put
\[
 d=d_{I,u,j},\qquad R_q=2^{q/2}r^{-j}.
\]
There is a partition into sets $C_1,\ldots,C_L$ and a possible remainder
$R$, where $L\le m-1$, such that each $C_l$ has a center $w_l\in C_l$ and
\begin{equation}\label{eq:greedysmall}
 C_l\subset B_d(w_l,R_q).
\end{equation}
If $R\ne\varnothing$, then for every $t\in R$ and every nonempty
$D\subset R\cap B_d(t,R_q/K)$,
\begin{equation}\label{eq:greedygain}
 F_{I,u,j}(A)\ge F_{I,u,j}(D)+\eta_{\rm G}2^q r^{-j}.
\end{equation}
\end{lemma}

\begin{proof}
Put $h=r^{-j}$ and $\epsilon=c_{\rm s}2^qh/12$. Starting with $S_1=A$,
choose $w_l\in S_l$ such that
\[
 F_{I,u,j}(S_l\cap B_d(w_l,R_q/K))
 \ge\sup_{w\in S_l}F_{I,u,j}(S_l\cap B_d(w,R_q/K))-\epsilon.
\]
Set $C_l=S_l\cap B_d(w_l,R_q)$ and $S_{l+1}=S_l\setminus C_l$.
Stop if the remaining set is empty; otherwise continue through $l=m-1$
and put $R=S_m$. This proves~\eqref{eq:greedysmall} and the cardinality
assertion. Approximate maximizers suffice, so no compactness assumption is
needed.

Suppose $R\ne\varnothing$, fix $t\in R$, and let $D$ be as in the
statement. Take $w_m=t$ and
\[
 H_l=\varphi_{I,u,j}(S_l\cap B_d(w_l,R_q/K))\quad(l<m),
 \qquad H_m=\varphi_{I,u,j}(D).
\]
The coefficient vectors at $w_1,\ldots,w_m$ have pairwise distances greater
than $R_q$ and have sup-norm at most $3h$. Each $H_l$ lies within distance
$R_q/K$ of the corresponding center. Since $t\in S_l$ and
$D\subset S_l\cap B_d(t,R_q/K)$ for $l<m$, the choice of $w_l$ gives
\[
 \min_{l\le m}b(H_l)\ge F_{I,u,j}(D)-\epsilon.
\]
For either allowed value of $m$,
\[
 (\log2)2^{q-2}\le\log m\le(\log2)2^{q-1},\qquad
 \min\{R_q\sqrt{\log m},R_q^2/(3h)\}=\tfrac13 2^qh.
\]
Apply~\eqref{eq:growth} and choose
$K\ge\max\{1,12C_{\rm s}/c_{\rm s}\}$. The concentration error is at most
$c_{\rm s}2^qh/12$, as is $\epsilon$. Thus~\eqref{eq:greedygain} holds
with $\eta_{\rm G}=c_{\rm s}/6$.
\end{proof}

\subsection{Choice of parameters}\label{sec:parameters}

Fix $K$ as in Lemma~\ref{lem:greedy}. Choose a sufficiently large positive
integer $\ell$, and set
\begin{equation}\label{eq:parameters}
 r=2^{2\ell},\qquad \kappa=3\ell,\qquad \alpha=2^{1/3}.
\end{equation}
Thus
\begin{equation}\label{eq:parameteridentities}
 2^\kappa=r^{3/2},\qquad (2/\alpha)^\kappa=r,
 \qquad 1<\alpha<2,\qquad r>2\alpha.
\end{equation}
We take $\ell$ large enough that
\begin{equation}\label{eq:rchoice}
 2C_*\sqrt{\log2}\,r^{-1/4}\le\frac{c_*}{24K^2},
\end{equation}
where $c_*,C_*$ are from Lemma~\ref{lem:removal}. Set
\begin{equation}\label{eq:gainconstants}
 \eta_{\rm W}=\frac{c_*}{24K^2},\qquad
 \eta=\min\{\eta_{\rm G},\eta_{\rm W}\}.
\end{equation}
All these parameters are now fixed universal constants. The constant $M$
in Theorem~\ref{thm:criterion} plays no role in their choice.

\subsection{The second subdivision}

\begin{lemma}\label{lem:second}
Fix $I,u,j,v$ and $J$ as in Lemma~\ref{lem:split}. Let $n\ge1$ and
$2\le k\le\kappa$. Suppose $B$ is nonempty and countable, $w\in B$, and
\begin{equation}\label{eq:finesmall}
 \sup_{t\in B}d_{J,v,j+1}(t,w)^2
          \le2^{n+k}r^{-2(j+1)}.
\end{equation}
There is a partition of $B$ into at most $N_n$ sets $D_p$ and a possible
remainder $E$ such that each $D_p$ has a point $t^p\in B$ with
\begin{equation}\label{eq:oldsmall}
 D_p\subset B_{d_{I,u,j}}(t^p,K^{-1}2^{n/2}r^{-j}).
\end{equation}
If $E\ne\varnothing$, then
\begin{equation}\label{eq:secondgain}
 F_{I,u,j}(B)\ge F_{J,v,j+1}(E)+\eta_{\rm W}2^n r^{-j}.
\end{equation}
Here it suffices to assume $F_{I,u,j}(B)<\infty$.
\end{lemma}

\begin{proof}
Use $f_i,g_i,\theta_i$ from Lemma~\ref{lem:split} and form the centered
three-block map
\begin{equation}\label{eq:phimap}
 \Phi(t)=\big((f_i(t_i))_{i\in I\setminus J},
       (\theta_i(t_i))_{i\in J},(g_i(t_i)-g_i(w_i))_{i\in J}\big).
\end{equation}
Let $y$ denote its last block. Then $b_\Phi(B)\le F_{I,u,j}(B)$ and
$b_y(E)=F_{J,v,j+1}(E)$. The full canonical distance remains $\bar d$.
Every coordinate increment of $\Phi$ is at most $6r^{-j}$, and
\eqref{eq:finesmall} gives
\[
 \diam_2(y(B))\le2\delta,\qquad
 \delta=2^{(n+k)/2}r^{-j-1}
       \le2^{n/2}r^{-j}r^{-1/4},
\]
where the last inequality uses $2^k\le2^\kappa=r^{3/2}$.

Apply Lemma~\ref{lem:removal} with
\[
 m=N_n,\qquad \rho=K^{-1}2^{(n-1)/2}r^{-j},\qquad
 \beta=6r^{-j}.
\]
Since $\log m=(\log2)2^n$ and $K\ge1$,
\[
 \min\{\rho\sqrt{\log m},\rho^2/\beta\}
       =\frac{2^n r^{-j}}{12K^2}.
\]
Consequently, either the $m$ balls cover $B$, or their remainder $E$ satisfies
\begin{align*}
 F_{I,u,j}(B)
 &\ge F_{J,v,j+1}(E)
   +\left(\frac{c_*}{12K^2}
                -2C_*\sqrt{\log2}\,r^{-1/4}\right)2^n r^{-j}\\
 &\ge F_{J,v,j+1}(E)+\eta_{\rm W}2^n r^{-j}.
\end{align*}
Disjointify the balls in any order to obtain the sets $D_p$. By
\eqref{eq:metriccomparison}, a $\bar d$-ball of radius $\rho$ is contained
in a $d_{I,u,j}$-ball of radius $\sqrt2\rho$, giving~\eqref{eq:oldsmall}.
\end{proof}

\section{The partition construction}\label{sec:construction}

Assume for now that $T$ is countable, $0\in T$, and $0<b(T)<\infty$.
The elementary estimate
\begin{equation}\label{eq:diameter}
 \diam_2(T)\le4b(T)
\end{equation}
follows by applying the $L^1$ Khintchine inequality to $X_s-X_t$ and using
$\E\max\{X_s,X_t\}=\frac12\E|X_s-X_t|$; see
\cite[Lemma~2.1]{BL2014}. Choose $j_0\in\Z$ such that
\begin{equation}\label{eq:jzero}
 r^{-j_0}\le\diam_2(T)\le r^{-j_0+1}.
\end{equation}
Start with
\[
 \A_0=\{T\},\qquad j_0(T)=j_0,\qquad \pi_0(T)=0,\qquad I_0(T)=\N.
\]

\Needspace{5\baselineskip}
We construct refining partitions and their data with the following
properties:
\begin{enumerate}[label=\textup{(P\arabic*)},leftmargin=2.8em]
\item\label{P:levels} $j_q-j_{q-1}\in\{0,1\}$ on every branch.
\item\label{P:centers} If the level is unchanged, so are the center and
the coordinate set. If it increases, the new center belongs to the parent
and
\begin{equation}\label{eq:indexupdate}
 I_q(t)=\{i\in I_{q-1}(t):
 |\pi_q(t)_i-\pi_{q-1}(t)_i|\le r^{-j_{q-1}(t)}\}.
\end{equation}
\item\label{P:energy} For every $A\in\A_q$ and $t\in A$,
\begin{equation}\label{eq:inductiveenergy}
 \norm{\varphi_{I_q(A),\pi_q(A),j_q(A)}(t)}_2^2
               \le2^q r^{-2(j_q(A)-1)}.
\end{equation}
\end{enumerate}
The initial energy estimate follows from $0\in T$ and~\eqref{eq:jzero}.
Property~\ref{P:centers} implies that $I_q$ is exactly the set defined
by~\eqref{eq:tracking}.

Suppose the construction has reached level $q-1$, and fix a parent
$A\in\A_{q-1}$. Its entire preceding history is constant on $A$.
We use the second rule below precisely when there is an integer $n$ with
\begin{equation}\label{eq:trigger}
 n\ge1,\quad 2\le q-n\le\kappa,\quad
 j_{n-1}=j_n,
 \qquad j_n+1=j_{n+1}=\cdots=j_{q-1}.
\end{equation}
Such an $n$, if it exists, is unique: $n+1$ is the last increase of the
level before $q$. If there is no such $n$, we use the first rule. This
convention makes the two cases exhaustive.

\subsection[First rule]{First rule: no index satisfies~\eqref{eq:trigger}}

Write $(I,u,j)=(I_{q-1}(A),\pi_{q-1}(A),j_{q-1}(A))$ and apply
Lemma~\ref{lem:greedy} with $m=N_{q-1}$. On the remainder, if present,
retain $I,u,j$. For each small piece $C_l$ with center $w_l$, assign
\begin{equation}\label{eq:ordinaryupdate}
 j_q(C_l)=j+1,\qquad \pi_q(C_l)=w_l,\qquad
 I_q(C_l)=\{i\in I:|w_{l,i}-u_i|\le r^{-j}\}.
\end{equation}
There are at most $N_{q-1}$ children.

\subsection[Second rule]{Second rule: an index satisfies~\eqref{eq:trigger}}

Put $k=q-n$ and write
\[
 (I,u,j)=(I_n,\pi_n,j_n),\qquad
 (J,v,j+1)=(I_{q-1},\pi_{q-1},j_{q-1}).
\]
By the unchanged levels in~\eqref{eq:trigger},
\begin{equation}\label{eq:historyparameters}
 v=\pi_{n+1},\qquad
 J=\{i\in I:|v_i-u_i|\le r^{-j}\}.
\end{equation}
First apply Lemma~\ref{lem:greedy} to $A$ with its current parameters
$(J,v,j+1)$ and with $s=N_{q-2}$ pieces allowed. On its remainder $R$,
retain the current data $(J,v,j+1)$.

Each small piece $B=C_l$, with center $w_l\in B$, satisfies
\[
 \sup_{t\in B}d_{J,v,j+1}(t,w_l)^2\le2^q r^{-2(j+1)}.
\]
Apply Lemma~\ref{lem:second} to $B$, using the old parameters $(I,u,j)$,
the current parameters $(J,v,j+1)$, and the index $n$. This partitions $B$
into at most $N_n$ sets $D_{l,p}$ and a possible remainder $E_l$.
Assign the same new data to every nonempty one of these children:
\begin{equation}\label{eq:secondupdate}
 j_q=j+2,\qquad \pi_q=w_l,\qquad
 I_q=J_l:=\{i\in J:|w_{l,i}-v_i|\le r^{-j-1}\}.
\end{equation}
The center need not belong to each child; it belongs to the parent $A$,
as required by Theorem~\ref{thm:criterion}.

Since $q-n\ge2$, we have $N_n\le s$. The number of children of $A$ is at
most
\begin{equation}\label{eq:childcount}
 1+(s-1)(N_n+1)\le1+(s-1)(s+1)=s^2=N_{q-1}.
\end{equation}

\subsection{Verification of the construction}

Both rules split each parent into at most $N_{q-1}$ children. Consequently
$|\A_q|\le N_{q-1}^2=N_q$, and the sequence is admissible.
Properties~\ref{P:levels} and~\ref{P:centers} follow directly from the
assignments.

On an unchanged-level child, the energy bound at level $q-1$ implies
\eqref{eq:inductiveenergy} at level $q$. For an increasing child of the
first rule, Lemma~\ref{lem:split} gives
\[
 \norm{\varphi_{I_q,w_l,j+1}(t)}_2^2
 \le d_{I,u,j}(t,w_l)^2\le2^q r^{-2j}.
\]
This is~\eqref{eq:inductiveenergy} with new level $j+1$. For an increasing
child of the second rule, the same lemma, now applied to $(J,v,j+1)$ and
$w_l$, gives
\[
 \norm{\varphi_{J_l,w_l,j+2}(t)}_2^2
 \le d_{J,v,j+1}(t,w_l)^2\le2^q r^{-2(j+1)},
\]
which is the required bound with new level $j+2$.

We record two further consequences. Fix $t\in T$ and freeze the parameters
along its branch by setting, for nonempty $D\subset T$,
\begin{equation}\label{eq:branchfunctional}
 \F_m(D)=F_{I_m(t),\pi_m(t),j_m(t)}(D),\qquad
 F_m=\F_m(A_m(t)).
\end{equation}
Repeated use of Lemma~\ref{lem:split} shows that, for every $D$,
\begin{equation}\label{eq:subsetmonotonicity}
 \F_{m+1}(D)\le\F_m(D).
\end{equation}
Together with set inclusion and~\eqref{eq:boundedfunctional}, this gives
\begin{equation}\label{eq:branchmonotonicity}
 0\le F_{m+1}\le F_m\le b(T).
\end{equation}

Finally, in either rule an unchanged-level child is the remainder of the
first subdivision. Thus, whenever $j_q(t)=j_{q-1}(t)$, for every $z\in A_q(t)$
and every nonempty
\[
 D\subset A_q(t)\cap
 B_{d_{I_{q-1},\pi_{q-1},j_{q-1}}}
       (z,K^{-1}2^{q/2}r^{-j_q(t)}),
\]
Lemma~\ref{lem:greedy} gives
\begin{equation}\label{eq:residualproperty}
 F_{q-1}\ge\F_{q-1}(D)+\eta_{\rm G}2^q r^{-j_q(t)}.
\end{equation}
This estimate will be used at the first level of each selected interval.

\section{Summing the scales}\label{sec:summation}

We first recall the ordering lemma used in~\cite[Lemma~6.1]{BL2014}.
Summing the geometric series exactly gives the constant below.

\begin{lemma}[Selection of scales]\label{lem:ordering}
Let $(a_n)_{n\ge0}$ be a bounded sequence of positive numbers and
$\alpha>1$. Define
\[
 \Pset=\{n\ge0:a_m<a_n\alpha^{|m-n|}\text{ for all }m\ne n\}.
\]
Then $\Pset\ne\varnothing$ and
\begin{equation}\label{eq:ordering}
 \sum_{n\ge0}a_n\le\frac{\alpha+1}{\alpha-1}\sum_{n\in\Pset}a_n.
\end{equation}
\end{lemma}

\begin{proof}
Put $n\preceq m$ when $a_m\ge a_n\alpha^{|m-n|}$. This is a partial
order, and its maximal elements are exactly $\Pset$. A strictly increasing
chain in this order multiplies the values by at least $\alpha$ at each step.
Boundedness therefore implies that every index is below a maximal element.
Consequently
\[
 \sum_{n\ge0}a_n
 \le\sum_{m\in\Pset}a_m\sum_{n\ge0}\alpha^{-|n-m|}
 \le\left(1+2\sum_{l\ge1}\alpha^{-l}\right)\sum_{m\in\Pset}a_m,
\]
which is~\eqref{eq:ordering}.
\end{proof}

Fix a branch $t\in T$ of the construction in Section~\ref{sec:construction},
and abbreviate
\[
 a_m=2^m r^{-j_m(t)},\qquad F_m=\F_m(A_m(t)).
\]
The sequence $(a_m)$ is bounded. Indeed, if $j_m=j_{m-1}$, take $D=\{t\}$
in~\eqref{eq:residualproperty}. Since $\F_{m-1}(\{t\})=0$,
\[
 a_m\le\eta_{\rm G}^{-1}F_{m-1}\le\eta_{\rm G}^{-1}b(T).
\]
If $j_m=j_{m-1}+1$, then $a_m=(2/r)a_{m-1}\le a_{m-1}$. Induction yields
\begin{equation}\label{eq:abounded}
 \sup_m a_m\le\max\{a_0,\eta_{\rm G}^{-1}b(T)\}.
\end{equation}
We may thus apply Lemma~\ref{lem:ordering} with the fixed value
$\alpha=2^{1/3}$ from~\eqref{eq:parameters}.

\begin{lemma}\label{lem:selecteddecrease}
For every $n\in\Pset\setminus\{0\}$,
\begin{equation}\label{eq:peakpattern}
 j_{n-1}=j_n<j_{n+1}=j_n+1.
\end{equation}
Moreover, the index
\begin{equation}\label{eq:taudef}
 \tau(n)=\min\{q\ge n+2:j_q=j_n+2\}
\end{equation}
exists, satisfies $\tau(n)\le n+\kappa$, and obeys
\begin{equation}\label{eq:selecteddecrease}
 F_{n-1}-F_{\tau(n)}\ge\eta a_n.
\end{equation}
\end{lemma}

\begin{proof}
If $j_n=j_{n-1}+1$, then $a_{n-1}=(r/2)a_n>\alpha a_n$, contradicting
$n\in\Pset$. If $j_{n+1}=j_n$, then $a_{n+1}=2a_n>\alpha a_n$, giving
the same contradiction. This proves~\eqref{eq:peakpattern}.

If the second increase has not occurred by $n+\kappa$, then
$j_{n+\kappa}=j_n+1$, and~\eqref{eq:parameteridentities} gives
\[
 \alpha^{-\kappa}a_{n+\kappa}
 =\frac{(2/\alpha)^\kappa}{r}a_n=a_n.
\]
This contradicts the strict inequality in the definition of $\Pset$.
Thus $q:=\tau(n)\le n+\kappa$.

At level $q$, condition~\eqref{eq:trigger} holds with this $n$, so the
second rule is used. Write $(I,u,j)$ for the data at level $n$ and
$(J,v,j+1)$ for the data at level $q-1$. Let $B=C_l$ be the small piece of
the first subdivision which contains $A_q(t)$. Such a piece exists because
the level increases at $q$; hence the branch does not enter its unchanged-level
remainder. There are two possibilities.

If $A_q(t)=D_{l,p}$ is one of the ball pieces, then
Lemma~\ref{lem:second} gives a point $t^p\in B\subset A_n(t)$ with
\[
 A_q(t)\subset A_n(t)\cap
 B_{d_{I,u,j}}(t^p,K^{-1}2^{n/2}r^{-j}).
\]
Because $j_{n-1}=j_n$, the parameters at levels $n-1$ and $n$ agree.
Apply~\eqref{eq:residualproperty} at level $n$ and then
\eqref{eq:subsetmonotonicity} to obtain
\[
 F_{n-1}\ge\F_{n-1}(A_q(t))+\eta_{\rm G}a_n
            \ge F_q+\eta_{\rm G}a_n.
\]

If $A_q(t)=E_l$ is the remainder of the second subdivision, then
$B\subset A_{n-1}(t)$, so~\eqref{eq:secondgain} gives
\[
 F_{n-1}\ge F_{I,u,j}(B)
 \ge F_{J,v,j+1}(E_l)+\eta_{\rm W}a_n
 \ge F_q+\eta_{\rm W}a_n.
\]
The last step is again the contraction comparison. Both cases imply
\eqref{eq:selecteddecrease} by~\eqref{eq:gainconstants}.
\end{proof}

\begin{proposition}\label{prop:sum}
The partitions constructed in Section~\ref{sec:construction} satisfy
\eqref{eq:targetsum}.
\end{proposition}

\begin{proof}
List the positive elements of $\Pset$ as $n_1<n_2<\cdots$, stopping if
the list is finite. Between $n_i+1$ and $\tau(n_i)-1$, the level equals
$j_{n_i}+1$. A later selected index cannot occur before $\tau(n_i)-1$,
because~\eqref{eq:peakpattern} requires an increase immediately afterwards.
Selected indices cannot be adjacent, again by~\eqref{eq:peakpattern}.
Therefore, whenever the indices on the right exist,
\begin{equation}\label{eq:overlap}
 \tau(n_i)\le n_{i+1}+1\le n_{i+2}-1.
\end{equation}
The intervals $[n_i-1,\tau(n_i)]$ with $i$ odd thus have disjoint
interiors, as do those with $i$ even. Monotonicity and nonnegativity of
$(F_m)$ show, first for finite partial sums, that
\[
 \sum_i\bigl(F_{n_i-1}-F_{\tau(n_i)}\bigr)\le2F_0.
\]
In conjunction with Lemma~\ref{lem:selecteddecrease}, this yields
\begin{equation}\label{eq:peakbound}
 \sum_{n\in\Pset\setminus\{0\}}a_n\le\frac{2F_0}{\eta}.
\end{equation}
The ordering lemma now gives
\begin{align*}
 \sum_{m\ge0}a_m
 &\le\frac{\alpha+1}{\alpha-1}
          \left(a_0+\frac{2F_0}{\eta}\right)\\
 &\le\frac{\alpha+1}{\alpha-1}
          \left(4+\frac2\eta\right)b(T).
\end{align*}
Here $a_0=r^{-j_0}\le\diam_2(T)\le4b(T)$ by
\eqref{eq:jzero} and~\eqref{eq:diameter}, and $F_0\le b(T)$ by
\eqref{eq:branchmonotonicity}. The constants are independent of the branch,
so taking the supremum over $t\in T$ proves~\eqref{eq:targetsum}.
\end{proof}

\section{Proof of the decomposition theorem}\label{sec:mainproof}

We first finish the argument for countable $T$. If $b(T)=0$,
\eqref{eq:diameter} shows that $T$ is a singleton, and the conclusion is
immediate. Otherwise translate $T$ so that $0\in T$. Translation does not
change $b(T)$ or $\gamma_2(T)$, and the translation can be added back to
the second component of a decomposition.

Use the partitions of Section~\ref{sec:construction} and put $M=r^2$.
The initial diameter condition~\eqref{eq:initialdiam} follows
from~\eqref{eq:jzero}. Properties~\ref{P:levels} and~\ref{P:centers}
give the required alternatives for the levels and centers, and identify
the coordinate sets with~\eqref{eq:tracking}. Finally, for every real $x$,
\[
 \min\{x^2,r^{-2j}\}\le\clip{x}{j}^{\,2}.
\]
The energy estimate~\eqref{eq:inductiveenergy} therefore implies
\eqref{eq:criterionenergy}, since
$2^n r^{-2(j_n-1)}=r^2 2^n r^{-2j_n}$.
Theorem~\ref{thm:criterion} and Proposition~\ref{prop:sum} give
\eqref{eq:main}, with universal constants because $r$ was fixed before
the construction.

For completeness, let $T\subset\ell^2$ now be arbitrary. Choose a
countable norm-dense subset $S\subset T$. Then $b(S)\le b(T)$, so the
countable case gives
\[
 S\subset S_1+S_2,\qquad
 \sup_{x\in S_1}\norm{x}_1\le Cb(T),\qquad
 \gamma_2(S_2)\le Cb(T).
\]
Finiteness of~\eqref{eq:gamma} implies that $S_2$ is totally bounded:
an admissible sequence whose sum is bounded by a finite number $L$ gives
finite covers of radius $L2^{-n/2}$. Hence its norm closure $\mathcal K$ is compact.
Also
\begin{equation}\label{eq:closuregamma}
 \gamma_2(\mathcal K)\le\gamma_2(S_2).
\end{equation}
Indeed, for fixed admissible nets, the sum in~\eqref{eq:gamma} is a
lower semicontinuous function of the point, as a supremum of its continuous
partial sums. Its uniform bound on $S_2$ therefore extends to $\mathcal K$, and
one can then take the infimum over the nets.

Given $t\in T$, take $s_m\in S$ with $s_m\to t$, and write
$s_m=x_m+y_m$ with $x_m\in S_1$ and $y_m\in S_2$. By compactness, a
subsequence of $y_m$ converges in $\ell^2$ to a point $y\in\mathcal K$.
Along this subsequence, $x_m\to t-y$. Coordinatewise convergence and
Fatou's lemma imply
\[
 \norm{t-y}_1\le\liminf_m\norm{x_m}_1\le Cb(T).
\]
Thus the sets
\[
 T_1=\{x\in\ell^2:\norm{x}_1\le Cb(T)\},\qquad T_2=\mathcal K
\]
satisfy $T\subset T_1+T_2$ and~\eqref{eq:main}, using
\eqref{eq:closuregamma}. This proves Theorem~\ref{thm:main}.
\qed

\begin{remark}
The construction uses only three current parameters $(I,u,j)$ at each
partition element, together with the finite history needed to recognize
\eqref{eq:trigger}. The coordinate-removal principle and the deterministic
partition criterion are the results from~\cite{BL2014} used here. The
change is in the truncation and partition construction: the metric
comparison~\eqref{eq:metriccomparison} makes the second subdivision
possible, and~\eqref{eq:overlap} accounts for all decreases in the chaining
sum.
\end{remark}

\end{document}